\documentclass[10 pt]{amsart}
\usepackage{amsfonts,amsmath,amsthm,amssymb} 

\usepackage[T2A,OT4]{fontenc}
\usepackage[utf8]{inputenc}
\usepackage{enumerate}
\usepackage{esvect}
\usepackage{mathtools}
\usepackage{bm}
\usepackage[svgnames]{xcolor}
\usepackage{float}
\usepackage{tikz}

\theoremstyle{plain}
\newtheorem{theorem}{Theorem}[section]

\newtheorem{lemma}[theorem]{Lemma}

\theoremstyle{definition}
\newtheorem{definition}{Definition}[section]

\theoremstyle{remark}

\numberwithin{equation}{section}

\begin{document}
\title{Inequalities of Levin-Ste\v{c}kin, Clausing and Chebyshev revisited}

\author[A. Witkowski]{Alfred Witkowski}
\email{alfred.witkowski@utp.edu.pl}
\date{4.11.2016}

\begin{abstract}
We prove the Levin-Ste\v{c}kin inequality using Chebyshev's inequality and symmetrization.
Symmetry and slightly modified Chebyshev's inequality are also the key to an elementary proof of Clausing's inequality .
\end{abstract}
\maketitle
\section{Introduction}
It seems that the Levin-Ste\v{c}kin inequality appeared first in an appendix to the Russian edition  of the famous Hardy, Littlewood and P\'olya's Bible on inequalities \cite{HLP_Rus}. The translator (Levin)  enumerates the appendices written by Ste\v{c}kin,  by Levin and by both of them. The inequality we consider here comes from Appendix I written by Ste\v{c}kin. But the English version of the appendix \cite{HLP_APP-En} did not probably make this distinction clear enough, so all inequalities cited in the literature are called Levin-Ste\v{c}kin.

\begin{theorem}[Levin-Ste\v{c}kin's inequality]\label{thm:LS}
Let the function $p\colon (0,1)\to\mathbb{R}$ satisfes the conditions
\begin{enumerate}
	\item $p$ is non-decreasing in $\left(0,\tfrac{1}{2}\right)$,
	\item $p$ is symmetric, i.e. $p(x)=p(1-x)$,
\end{enumerate}
then for every convex function $\varphi$ the inequality
\begin{equation}
\int\limits_0^1p(x)\varphi(x)dx\leq \int\limits_0^1p(x)dx\int\limits_0^1\varphi(x)dx.
\label{eq:LS}
\end{equation}
\end{theorem}
The original proof is elementary, but quite complicated. Recently Mercer (\cite{Mercer}) published a proof that uses the notion of extremal points of the set of concave positive  functions satisfying $\int_0^1 f(x)dx\leq 1$. His method, not very elementary, has an advantage: leads to a simple proof of the Clausing inequality.
\begin{theorem}[Clausing's inequality \cite{Cla}]

	Let $p$ be   nonnegative functions on $(0,1)$ satisfying the following conditions:
	\begin{itemize}
		\item $p$  are symmetric (i.e. $p(x)=p(1-x)$),
		\item $p$ is non-decreasing on $[0,1/2]$,
		\end{itemize}
	Then for every  concave, positive function $\varphi$  the inequality
	\begin{equation}
	\int_0^1 p(x)\varphi(x)dx\leq \int_0^1 \varphi(x)dx\int_0^1 4\min\{x,1-x\}p(x)dx
	\end{equation}
	holds.
 \end{theorem}
Both inequalities make the reader think of the inequality of Chebyshev, linking the integral of a product of functions with the product of integrals. 
\begin{theorem}[Chebyshev's inequality]
If the functions $f,g:[a,b]\to\mathbb{R}$ are monotone in the same direction, then
$$\frac{1}{b-a}\int_a^b f(x)g(x)dx\geq \frac{1}{b-a}\int_a^b f(x)dx\frac{1}{b-a}\int_a^b g(x)dx$$
The inequality is reversed if the monotonicities are opposite.
\end{theorem}

Our aim is to give elementary proofs of Levin-Ste\v{c}kin's and Clausing's inequalities. The proofs we offer here are sponsored by the word \textit{symmetrization}.

\section{The Levin-Ste\v{c}kin inequality}
We prove this inequality in two steps: firstly we show that Theorem \ref{thm:LS} is valid for symmetric functions:
\begin{lemma}\label{lem:LS_sym}
	Under the assumptions of Theorem \ref{thm:LS} if $\varphi$ is symmetric and convex, then the inequality \eqref{eq:LS} holds. 
\end{lemma}
\begin{proof}
	Suppose $\varphi$ is convex. Its symmetry implies that it is non-increasing in the interval $\left(0,\tfrac{1}{2}\right)$, and using Chebyshev's inequality we get
	\begin{align*}
	\int\limits_0^1p(x)dx\int\limits_0^1\varphi(x)dx	&=\left(\int\limits_0^{1/2}p(x)dx+\int\limits_{1/2}^1p(x)dx\right)\left(\int\limits_0^{1/2}\varphi(x)dx+\int\limits_{1/2}^1\varphi(x)dx\right)	\\
		&	=4\int\limits_0^{1/2}p(x)dx\int\limits_0^{1/2}\varphi(x)dx\geq 2\int\limits_0^{1/2}p(x)\varphi(x)dx=\int\limits_0^{1}p(x)\varphi(x)dx.\qedhere
	\end{align*}
\end{proof}
Now consider arbitrary $\varphi$.
\begin{proof}[Proof of the Levin-Ste\v{c}kin inequality]
Once more we shall explore the symmetry. Note that for convex $\varphi$ the function $\frac{\varphi(x)+\varphi(1-x)}{2}$ is convex and symmetric, so we can use Lemma \ref{lem:LS_sym} 
\begin{align*}
	\int\limits_0^{1}p(x)\varphi(x)dx&=\int\limits_0^{1}p(x)\frac{\varphi(x)+\varphi(1-x)}{2}dx	\\
	&	\leq \int\limits_0^{1}p(x)dx\int\limits_0^{1}\frac{\varphi(x)+\varphi(1-x)}{2}dx\\
	&=\int\limits_0^{1}p(x)dx\int\limits_0^{1}\varphi(x)dx.
\end{align*}
\end{proof}

\section{Chebyshev's inequality}
To prove the Clausing inequality we need a bit stronger version of  Chebyshev's inequality, where the monotonicity of one function get replaced by a weaker condition. Note that this result is somewhat similar to the result of Brunn \cite{Bru}.
\begin{definition}
	We shall say that an integrable function $f:[a,b]\to\mathbb{R}$ belongs to the class $M^+$ if there is a $c\in [a,b]$ such that
	\begin{enumerate}
		\item if $f(x)<\frac{1}{b-a}\int_a^b f(x)dx$, then $x<c$ and \label{enum:one}
		\item if $f(x)>\frac{1}{b-a}\int_a^b f(x)dx$, then $x>c$. \label{enum:two}
	\end{enumerate}
We say that $f$ belongs to $M^-$ if the inequalities in \eqref{enum:one} and \eqref{enum:two} are reversed.
\end{definition}

Obviously every non-decreasing function belongs to the class $M^+$ an a non-increasing one is a member of $M^-$.
\begin{theorem}\label{thm:Chebyshev-AW}
	If $f,g:[a,b]\to\mathbb{R}$ are integrable, $g$ is non-decreasing and $f\in M^+$ or $g$ is non-increasing and $f\in M^-$, then
	$$\frac{1}{b-a}\int_a^b f(x)g(x)dx\geq \frac{1}{b-a}\int_a^b f(x)dx\frac{1}{b-a}\int_a^b g(x)dx.$$
	Exchanging $M^+$ and $M^-$ toggles the inequality.
\end{theorem}
\begin{proof}
	Let $f\in M^+$ and $g$ be non-decreasing (the proof in other cases is similar). Denote $f^*=\frac{1}{b-a}\int_a^b f(x)dx$. Then
	\begin{align*}
		\int_a^b \left[f(x)-f^*\right]g(x)dx&=\int_a^c \left[f(x)-f^*\right]g(x)dx+\int_c^b \left[f(x)-f^*\right]g(x)dx	\\
		&	\geq g(c)\int_a^c \left[f(x)-f^*\right]dx+g(c)\int_c^b \left[f(x)-f^*\right]dx=0.\qedhere
	\end{align*}
\end{proof}

\section{Clausing's inequality}

In this section we present an elementary proof of a generalization of the Clausing inequality. 
\begin{theorem}\label{thm:Clausing}
	Let $p,q$ be   nonnegative functions on $(0,1)$ satisfying the following conditions:
	\begin{itemize}
		\item $p$ and $q$ are symmetric (i.e. $p(x)=p(1-x)$),
		\item $p$ is increasing on $[0,1/2]$,
		\item $q$ is convex on $[0,1/2]$,
		\item $q(0)=0$ and $\int_0^1 q(x)dx=1$.
	\end{itemize}
	Then for every  concave function $\varphi$ with $\varphi(0)+\varphi(1)\geq 0$ the inequality
	\begin{equation}
	\int_0^1 p(x)\varphi(x)dx\leq \int_0^1 \varphi(x)dx\int_0^1 p(x)q(x)dx
	\label{ineq:Clausing}
	\end{equation}
	holds.
 \end{theorem}
\begin{proof}
	Assume first that $\varphi$ is symmetric and denote $\int_0^1\varphi(x)dx=K$. The inequality \eqref{ineq:Clausing} can be rewritten as
	\begin{equation}
	0\leq \int_0^{1/2} [K q(x)-\varphi(x)]p(x)dx.
	\label{ineq:Clausing2}
	\end{equation}
	The Hermite-Hadamard inequality yields $K\geq 0$ and the symmetry of $\varphi$ implies $\varphi(0)\geq 0$, thus the function $K q-\varphi$ is convex,   $K q(0)-\varphi(0)\leq 0$ and $\int_0^{1/2} [K q(x)-\varphi(x)]dx=0$, therefore it belongs to the class $M^+$, and by Theorem \ref{thm:Chebyshev-AW} $\int_0^{1/2} [K q(x)-\varphi(x)]p(x)dx	\ge 0$ 
	which proves \eqref{ineq:Clausing2}.
	
	Now let $\varphi$ be arbitrary. We have
	\begin{align*}
	\int_0^1 p(x)\varphi(x)dx	&=\int_0^1 p(x)\frac{\varphi(x)+\varphi(1-x)}{2}dx	\\
		&	\leq \int_0^1 \frac{\varphi(x)+\varphi(1-x)}{2}dx\int_0^1 p(x)q(x)dx\quad\text{(by \eqref{ineq:Clausing})}\\
		&= \int_0^1 \varphi(x)dx\int_0^1 p(x)q(x)dx
	\end{align*}
	which completes the proof.
\end{proof}

The function $q_0(x)=4\min\{x,1-x\}$ is a borderline between admissible $q$'s and sample functions $\varphi$. Setting $\varphi=q_0$ in \eqref{ineq:Clausing} we obtain
$$\int_0^1 p(x)q_0(x) dx\leq \int_0^1 p(x)q(x) dx$$
which means that $q_0$ provides the best bound in \eqref{ineq:Clausing}.

\end{document}